\documentclass[a4paper,10pt]{amsart}

\usepackage[UKenglish]{babel}

\usepackage[utf8]{inputenc}
\DeclareUnicodeCharacter{2009}{FIXME}

\usepackage{amssymb}
\usepackage{amsmath}
\usepackage{amsthm}
\usepackage{bbm}
\usepackage{verbatim, stackrel}
\usepackage{comment,cite}	

\usepackage{hyperref}
\hypersetup{
    colorlinks=true,
    linkcolor=blue,
    citecolor=red,
    filecolor=magenta,      
    urlcolor=cyan,
    pdftitle={Stability on AM-Spaces},
    linktocpage=true,
}

\usepackage[shortlabels]{enumitem}
\usepackage{marginnote}
\usepackage{color}

\newcommand{\bbN}{\mathbb{N}}

\newcommand{\bbR}{\mathbb{R}}

\DeclareMathOperator{\re}{Re} 
\newcommand{\dx}{\;\mathrm{d}} 
\newcommand{\norm}[1]{\left\lVert #1 \right\rVert} 
\DeclareMathOperator{\dom}{dom} 

\newcommand{\spb}{s} 
\newcommand{\gbd}{\omega_0} 

% ...

\theoremstyle{definition}
\newtheorem{definition}{Definition}
\newtheorem{remark}[definition]{Remark}
\newtheorem{remarks}[definition]{Remarks}

\theoremstyle{plain}
\newtheorem{theorem}[definition]{Theorem}

\numberwithin{equation}{section}

\begin{document}

\title[Stability of (eventually) positive semigroups on $C_0(L)$]{Stability of (eventually) positive semigroups on spaces of continuous functions}
\author{Sahiba Arora}
\address{Sahiba Arora, Technische Universität Dresden, Institut für Analysis, Fakultät für Mathematik , 01062 Dresden, Germany}
\email{sahiba.arora@mailbox.tu-dresden.de}
\author{Jochen Gl\"uck}
\address{Jochen Gl\"uck, Universität Passau, Fakultät für Informatik und Mathematik, 94032 Passau, Germany}
\email{jochen.glueck@uni-passau.de}
\date{\today}
\begin{abstract}
	We present a new and very short proof of the fact that, for positive $C_0$-semigroups on spaces of continuous functions, the spectral and the growth bound coincide. Our argument, inspired by an idea of Vogt, makes the role of the underlying space completely transparent and also works if the space does not contain the constant functions -- a situation in which all earlier proofs become technically quite involved.
	
	We also show how the argument can be adapted to yield the same result for semigroups that are only eventually positive rather than positive.
\end{abstract}

\maketitle

\subsection*{Positive semigroups}

A classical question in the theory of positive $C_0$-semigroups is whether, on a function space $E$, the spectral bound 
\begin{align*}
	\spb(A) := \sup \{\re \lambda: \, \lambda \text{ is a spectral value of } A\} \in [-\infty,\infty)
\end{align*}
of  $(e^{tA})_{t \geq 0}$ coincides with its growth bound
\begin{align*}
	\gbd(A) := \inf \{\omega \in \bbR: \, (e^{t(A-\omega)})_{t \geq 0} \text{ is bounded}\} \in [-\infty, \infty).
\end{align*}
Recall that the inequality $\spb(A)\leq \gbd(A)$ is true for all $C_0$-semigroups, and the question for which semigroups equality holds is relevant in order to derive the long term behaviour of a semigroup from spectral properties of their generator.

The answer to the above question is positive on spaces of continuous functions, but the proofs that are known for this result (for instance, in \cite[Theorem~B-IV-1.4]{Nagel1986} or \cite[Theorems~5.3.8]{ArendtBattyHieberNeubrander2011}) tend to be technical in case that the space does not contain the constant functions (a situation which occurs most prominently for the space $C_0(L)$ of continuous functions that vanish at infinity, on a locally compact Hausdorff space $L$). Specifically, the proofs rely on the structure of the dual space $E'$ and the corresponding semigroup $(e^{tA'})_{t \geq 0}$ which, in general, is not $C_0$. Because of this reason, the proofs employ the so-called \emph{sun-dual} semigroup.
	
The first purpose of our paper is to provide a short and transparent proof of this classical result.

\begin{theorem}
	\label{thm:positive}
	Let $(e^{tA})_{t\geq 0}$ be a positive $C_0$-semigroup on an AM-space $E$. Then $\spb(A)=\gbd(A)$.
\end{theorem}

We have used the following terminology in the theorem: an \emph{AM-space} $E$ is a Banach lattice with the additional property $\norm{\sup\{f, g\}} = \sup\{\norm{f}, \norm{g}\}$ for all $0 \le f,g \in E$. Typical spaces of continuous functions, such as $C(K)$ for a compact space $K$ or $C_0(L)$ for a locally compact space $L$, are AM-spaces.
For the proof of Theorem~\ref{thm:positive}, we only need that every non-empty relatively compact subset of an AM-space $E$ has a supremum in $E$ \cite[Proposition II.7.6]{Schaefer1974}.

\begin{proof}[Proof of Theorem~\ref{thm:positive}]
	By a rescaling argument, it suffices to show that $(e^{tA})_{t\geq 0}$ is bounded whenever $\spb(A) < 0$. Indeed, if this implication is true, then the rescaled semigroup $(e^{t(A-s)})_{t\geq 0}$ is bounded for all $s>\spb(A)$. Consequently, $\gbd(A)\leq s$ for all $s>\spb(A)$ which yields $\gbd(A)\leq \spb(A)$.
	
	So suppose that $\spb(A) < 0$ and let $0 \le f \in E$ be fixed. Since the positive cone spans $E$, we only need to show that the orbit of $f$ is bounded. Boundedness of the semigroup will then follow from the uniform boundedness theorem.
	
	By the property of AM-spaces mentioned before the proof, the compact set $\{e^{tA}f : t\in [0,1] \}$ has a supremum $g \ge 0$ in $E$. 
	
	Fix a time $t\geq 1$ and let $I= [t-1,t]$. 
	Then $e^{tA}f=e^{sA}e^{(t-s)A}f\leq e^{sA}g$ for all $s\in I$, thus
	\[
		0 \le e^{tA}f=\int_I e^{tA}f \dx s \leq \int_I e^{sA}g \dx s \leq \int_0^\infty e^{sA}g \dx s =: \tilde{g} \in E;
	\]
	here the last integral exists as an improper Riemann integral due to the positivity of the $C_0$-semigroup and the assumption $\spb(A)<0$ (see, for instance, \cite[Theorem~5.3.1 and Proposition~5.1.4]{ArendtBattyHieberNeubrander2011}). The result now follows using $\norm{e^{tA}f} \leq \norm{\tilde g}$ for all $t\geq 1$.
\end{proof}

The main idea that we used above is an adaptation of an argument of Vogt, who recently presented an intriguingly easy proof of $\spb(A) = \gbd(A)$ on $L^p$-spaces \cite{Vogt2021} (while earlier proofs due to Weis were much more involved \cite{Weis95,Weis98}).
For more historical details on the question whether $\spb(A) = \gbd(A)$ for positive semigroups, we refer for instance to \cite[p.\,389]{ArendtBattyHieberNeubrander2011}. Recently, some conditions were given in \cite[Section~4]{PrajapatiSinhaSrivastava2019} which ensure that the above equality holds for positive semigroups on non-commutative $L^p$-spaces.
For a deeper connection between this question and geometric properties of Banach spaces, see \cite[Section~5]{RozendaalVeraar2018}.

\begin{remark}
	\label{rem:order-bounded}
	The proof of Theorem~\ref{thm:positive} cannot directly be generalized to a larger class of Banach lattices since AM-spaces are the only Banach lattices in which every non-empty relatively compact set has a supremum 
	(see, for instance, \cite[p.\,275]{Wickstead1975}).
		
	However, the argument directly generalizes to ordered Banach spaces with closed, normal, and generating cones, under the assumption that every relatively compact subset is order bounded; 
	conditions for this latter property can be found in \cite[Theorem~1]{Wickstead1975}.
	For this reason, our argument also yields a new and simple proof of \cite[Theorem~4]{BattyDavies1983}. In particular, Theorem~\ref{thm:positive} is true on every ordered Banach space which contains an order unit and has normal cone.
\end{remark}

\begin{remarks}\label{rem:properties-used}
	Our proof of Theorem~\ref{thm:positive} uses two ingredients:
	\begin{enumerate}[(a)]
		\item 
		The fact that every relatively compact set in an AM-space has a supremum.
		
		For concrete functions spaces such as $C_0(L)$, where $L$ is locally compact, this is not difficult to show by employing the Arzelà--Ascoli theorem (and for general AM-spaces, the proof is essentially the same).
		For $C(K)$, where $K$ is compact, the situation is even simpler since every norm bounded set in $C(K)$ is order bounded, which -- as mentioned in Remark~\ref{rem:order-bounded} -- suffices for the proof of Theorem~\ref{thm:positive}.
		
		\item 
		The fact that, if $\spb(A) < 0$ and $g \in E$, then $\int_0^\infty e^{sA}g \dx s$ converges in $E$ as an improper Riemann integral. 
		
		This non-trivial result from the theory of positive semigroups is true on all Banach lattices (and even on ordered Banach spaces with a normal generating cone); see the above-quoted \cite[Theorem~5.3.1 and Proposition~5.1.4]{ArendtBattyHieberNeubrander2011}.
		Throughout the literature, proofs that $\spb(A) = \gbd(A)$ holds for all positive semigroups on particular classes of function spaces, require this result as an important ingredient. 
		The same is true for our proof; our main contribution is that we significantly simplify the rest of the argument, which is the only place where the specific structure of the space is used.
	\end{enumerate}
\end{remarks}

\subsection*{Eventually positive semigroups}

While the theory of positive semigroups can be considered a classical topic in analysis, many evolution equations have been recently discovered to only exhibit \emph{eventually positive behaviour} -- i.e., for positive initial values, the solution first changes sign but then becomes and stays positive. This gives rise to the following notion:

A $C_0$-semigroup $(e^{tA})_{t\geq 0}$ on a Banach lattice $E$ is said to be \emph{individually eventually positive} if for every positive initial datum $f\in E$, there exists $t_0\geq 0$ such that $e^{tA}f$ is also positive for all $t\geq t_0$. 
The semigroup is called \emph{uniformly eventually positive} if $t_0$ can be chosen to be independent of $f$.

Many examples of eventually positive semigroups occur in the study of concrete PDEs; see for instance \cite[Proposition~2 and Examples~6 and 7]{HusseinMugnolo2020}, \cite[Section~3]{AddonaGregorioRhandiTacelli2021}, \cite[Proposition~5.5]{BeckerGregorioMugnolo2021}, and \cite[Section~7]{DenkKunzePloss2021} for recently discovered examples.
The prevalence of eventually positive semigroups in concrete differential equations makes it a worthwhile goal to understand their behaviour at a general and theoretical level, an endeavour which started with the papers \cite{DanersGlueckKennedy2016a, DanersGlueckKennedy2016b};
general results about the spectrum and long-term behaviour of eventually positive semigroups have recently been obtained in \cite{AroraGlueck2021a}. 

In the rest of this article, we specifically study whether $\spb(A) = \gbd(A)$ in the eventually positive case. The property of positive $C_0$-semigroups mentioned in Remark~\ref{rem:properties-used}(b) surprisingly holds even if the semigroup is only individually eventually positive; see \cite[Proposition~7.1]{DanersGlueckKennedy2016a}. As a consequence, it was shown in \cite[Theorem~7.8]{DanersGlueckKennedy2016a} that on $L^1$- and $L^2$-spaces, as well as on $C(K)$, the property $\spb(A) = \gbd(A)$ is true for individually eventually positive semigroups. On $L^p$ for other values of $p$, the aforementioned argument of Vogt carries over to uniformly eventually positive semigroups \cite{Vogt2021}; the individually eventually positive case remains currently open on $L^p$ for $p \in (1, \infty) \setminus \{2\}$.
In the following, we settle the question for all AM-spaces (while the proof given for $C(K)$ in \cite[Theorem~7.8]{DanersGlueckKennedy2016a} does not work on spaces not containing an order unit). The following argument even works for individually eventually positive semigroups. 

\begin{theorem}
	\label{thm:ev-pos}
	Let $(e^{tA})_{t\geq 0}$ be an individually eventually positive $C_0$-semigroup on an AM-space $E$. Then $\spb(A)=\gbd(A)$.
\end{theorem}

	We point out that the argument given in \cite[Theorem~7.8]{DanersGlueckKennedy2016a} for the special case $E = C(K)$ assumes that the generator $A$ is real, while we make no such assumption. 
	Furthermore, as mentioned before Theorem~\ref{thm:positive}, the proofs known for positive semigroups on $C_0(L)$, for a locally compact Hausdorff space $L$, employ the sun-dual semigroup on the dual space $E'$. This works since the closure of the domain $\dom{A'}$ of the dual operator can be shown to be an ideal in $E'$ when the semigroup is positive. The proof does not carry over to eventually positive semigroups, and it is unclear whether the same ideal property is still true for them. 
	
	On the other hand, the proof that we presented for Theorem~\ref{thm:positive} can be modified for the eventually positive case, which is a major advantage of our approach. We are therefore able to provide an affirmative answer to a question posed after \cite[Theorem~7.8]{DanersGlueckKennedy2016a}.
	
\begin{proof}[Proof of Theorem~\ref{thm:ev-pos}]
	We adjust the proof of Theorem~\ref{thm:positive} to the present situation: Assume that $\spb(A)<0$ and fix a vector $0\leq f\in E$; again, it suffices to show that the orbit of $f$ is bounded. 
	
	We can find a time $t_f\geq 0$ such that $e^{tA}f\geq 0$ for all $t\geq t_f$, owing to the individual positivity assumption on the semigroup. As in the proof of the previous theorem, the supremum 
	$
		g:= \sup \{e^{tA} f : t \in [t_f,t_f+1]\}
	$
	exists, since $E$ is an AM-space. In addition, the vector $g$ is positive (being a supremum of positive vectors), so there exists $t_g\geq 0$ such that $e^{tA}g\geq 0$ for all $t\geq t_g$. 
	
	Note that the set $P:=\{g - e^{tA}f : t\in [t_f,t_f+1]\}$ is a compact subset of the positive cone in $E$, and due to the individual eventual positivity of the semigroup, it is covered by its closed subsets
	\[
		P_n:=\{h\in P: e^{tA}h\geq 0 \text{ for all }t\geq n\},
	\]
	where $n$ runs through $\bbN$. Hence, by the Baire category theorem, there exists $N\in\bbN$ such that $(P_N)^{\circ}$ is non-void; here, the interior is taken within the space $P$. 
	
	Let us now consider the map $\Phi : [t_f,t_f+1] \to P$ defined by $\Phi(t)= g- e^{tA}f$ for all $t\in [t_f,t_f+1]$. 
	It is both continuous and surjective, therefore the pre-image $J:= \Phi^{-1} \big((P_N)^\circ\big)$ is non-void and open in $[t_f,t_f+1]$. Consequently, there exist $t_0\geq 0$ and $0 < \ell \leq 1$ such that $[t_0,t_0+\ell] \subseteq J\subseteq [t_f,t_f+1]$.
	Letting $t_1=\max\{N,t_g\}$, we conclude that the vectors $e^{sA}\big(g- e^{tA}f\big)$ are positive for all $s\geq t_1$ and all $t\in [t_0,t_0+\ell]$. 
	
	Finally, fix $t\geq t_0+t_1+\ell$ and let $I:=[t-t_0-\ell,t-t_0] \subseteq [t_1,\infty)$. With this notation, we have $e^{tA}f=e^{sA}e^{(t-s)A}f\leq e^{sA}g$ for all $s\in I$, and so
	\[
		0 \leq e^{tA}f = \frac1\ell \int_I e^{tA}f\, ds \leq \frac1\ell \int_I e^{sA}g\, ds \leq \frac1\ell \int_{t_1}^\infty e^{sA}g\, ds =: \tilde g \in E;		
	\]
	here we used that the integral on the right exists as an improper Riemann integral since $(e^{tA})_{t\geq 0}$ is individually eventually positive and $\spb(A)<0$ \cite[Proposition~7.1]{DanersGlueckKennedy2016a}; moreover, for the first inequality we used that $t \ge t_0 \ge t_f$, and for the last inequality we used $t_1 \ge t_g$.
	
	We conclude that $\norm{e^{tA}f} \leq \norm{\tilde g}$ for all $t \geq t_0+t_1+\ell$, which proves that the orbit of $f$ is bounded.
\end{proof}

\subsection*{Acknowledgements} 

The authors are indebted to Hendrik Vogt for various fruitful discussions and for sharing an early version of \cite{Vogt2021} with them.
This paper was written during a pleasant visit of the first author to Universität Passau.	
The first named author was supported by 
Deutscher Aka\-de\-mi\-scher Aus\-tausch\-dienst (Forschungs\-stipendium-Promotion in Deutschland).

\bibliographystyle{plainurl}
\bibliography{literature}

\end{document}